\documentclass[a4paper,12pt,twoside,openright]{amsart}
\usepackage[english]{babel}
\usepackage[utf8]{inputenc}
\usepackage{fancyhdr}

\usepackage{amssymb}
\usepackage{amsmath} 
\usepackage{amsthm} 
\usepackage{enumitem}
\usepackage{array}
\usepackage[table]{xcolor}
\arrayrulecolor{blue}
\usepackage{tikz-cd}
\usetikzlibrary{fit}
\usepackage{extarrows}
\usepackage{mathtools}
\usepackage{faktor}
\usepackage{float}

\makeatletter
\def\subsubsection{\@startsection{subsubsection}{3}%
  \z@{.5\linespacing\@plus.7\linespacing}{-.5em}%
  {\normalfont\bfseries}}
\makeatother

\theoremstyle{plain} 
\newtheorem{thm}{Theorem}[section] 
\newtheorem{cor}[thm]{Corollary} 
\newtheorem{lem}[thm]{Lemma}

\theoremstyle{definition} 
\newtheorem{defn}{Definition}[] 
\newtheorem{es}{Example}

\theoremstyle{remark} 
\newtheorem{oss}{Remark}

\makeatletter 
\newcommand{\xhooklongrightarrow}[2][]{%
  \ext@arrow3399{\hooklongrightarrowfill@}{#1}{#2}} 
\makeatother

 \DeclareMathOperator{\hgt}{ht}

    \DeclareMathOperator{\Tr}{Tr}

   \DeclareMathOperator{\lt}{in_\prec}

 \DeclareMathOperator{\Min}{Min}

 \usepackage{blindtext}
\title{Knutson ideals of generic matrices\\}
\author{Lisa Seccia}
\address{Universit\`a  di Genova,  Dipartimento di Matematica. 
 Via Dodecaneso 35, 16146 Genova, Italy}
\email{seccia@dima.unige.it}
 \date{}
 
\begin{document}
 \maketitle
 \begin{abstract}
 In this paper we show that determinantal ideals of generic matrices are Knutson ideals. This fact leads to a useful result about Gr\"obner bases of certain sums of determinantal ideals.  More specifically, given $I=I_1+\ldots+I_k$ a sum of ideals of minors on adjacent columns or rows, we prove that the union of the Gr\"obner bases of the $I_j$'s is a Gr\"obner basis of $I$.
 
 \end{abstract}

 \section{Introduction}

 Let $\mathbb{K}$ be a field of any characteristic. Fix $f \in S= \mathbb{K}[x_1,\ldots,x_n]$ a polynomial such that its leading term $\lt (f)$ is a squarefree monomial  for some term order $\prec$. We can define many more ideals starting from the principal ideal $(f)$ and taking associated primes, intersections and sums. Thereby, if $\mathbb{K}$ has characteristic $p$, we obtain a family of ideals which are compatibly split with respect to $\Tr(f^{p-1}\bullet)$ (see \cite{Kn} for more details).\par
 Geometrically  this means that we start from the hypersurface defined by $f$ and we construct a family of new subvarieties $\{Y_i\}_i$ by taking irreducible components, intersections and unions. 
 
 \begin{defn}[Knutson ideals] \label{K.I.} Let $f \in S= \mathbb{K}[x_1,\ldots,x_n]$ be a polynomial such that its leading term $\lt (f)$ is a squarefree monomial  for some term order $\prec$ .
Define $\mathcal{C}_f$ to be the smallest set of ideals satisfying the following conditions:
\begin{enumerate}
\item[1.] $(f) \in \mathcal{C}_f$;
\item[2.]  If $I \in \mathcal{C}_f$ then $I:J \in \mathcal{C}_f$ for every ideal $J \subseteq S$;
\item[3.] If $I$ and $J$ are in $\mathcal{C}_f$ then also $I+J$ and $I \cap J$ must be in $\mathcal{C}_f$.
\end{enumerate} 
 \end{defn}
 
 This class of ideals has some interesting properties which were first proved by Knutson in the case $\mathbb{K}=\mathbb{Z}/p\mathbb{Z}$ and then generalized to fields of any characteristic in \cite{Se}:
 
\begin{itemize}
 \item[i)] Every $I \in \mathcal{C}_f$ has a squarefree initial ideal, so every Knutson ideal is radical.
 \item[ii)] If two Knutson ideals are different their initial ideals are different. So $\mathcal{C}_f$ is finite.
  \item[iii)] The union of the Gr\"obner bases of Knutson ideals associated to $f$ is a Gr\"obner basis of their sum.
\end{itemize}

 \begin{oss}
 Actually, assuming that every ideal of $\mathcal{C}_f$ is radical, the second condition in Definition \ref{K.I.} can be replaced by the following:
 \begin{itemize}
 \item[$2^\prime .$] If $I \in \mathcal{C}_f$ then $\mathcal{P} \in \mathcal{C}_f$ for every $\mathcal{P} \in \Min(I)$.
 \end{itemize}
 
 \end{oss}
 In this paper we will continue the study undertaken in \cite{Se} about Knutson ideals.\par

So far, it has been proved that determinanatal ideals of Hankel matrices are Knutson ideals for a suitable choice of $f$ ( \cite[Theorems 3.1,3.2]{Se}). As a consequence of these results, one can derive an alternative proof (see \cite[Corollary 3.3]{Se}) of the $F$-purity of Hankel determinantal rings, a result recently proved in \cite{CMSV}.\\

 In this paper we are going to show that also determinantal ideals of  generic matrices are Knutson ideals (see Theorem \ref{propgen2}). In particular, they define $F$-split rings. This was already known since the 1990's from a result by Hochster and Huneke (\cite{HH}). \par 
As a corollary we obtain an interesting result about Gr\"obner bases of certain sums of determinantal ideals.  More specifically, given $I=I_1+\ldots+I_k$ a sum of ideals of minors on adjacent columns or rows, we will prove that the union of the Gr\"obner bases of the $I_j$'s is a Gr\"obner basis of $I$ (see Corollary \ref{corsum}).\par

\begin{es}
Let $X=(x_{ij})$ be the generic square matrix of size 6 and consider the ideal
$$J= I_3(X_{[1,3]})+I_3(X^{[1,3]})$$
in the polynomial ring $S=\mathbb{K}[X]$. Then $J$ is the ideal generated by the 3-minors of the following highlighted ladder

\[ X= \left[ \: 
\begin{array}{>{\columncolor{blue!5}}c>{\columncolor{blue!5}}c>{\columncolor{blue!5}}cccc}
\rowcolor{blue!5}
x_{11} & x_{12}& x_{13} & x_{14} & x_{15}& x_{16} \\
\rowcolor{blue!5}
x_{21} & x_{22} & x_{23} & x_{24} & x_{25}& x_{26} \\
\rowcolor{blue!5}
x_{31} & x_{32} & x_{33} & x_{34} & x_{35}& x_{36} \\
x_{41} & x_{42} & x_{43} & x_{44} & x_{45}& x_{46} \\
x_{51} & x_{52} & x_{53} & x_{54} & x_{55}& x_{56} \\
x_{61} & x_{62} & x_{63}& x_{64}& x_{65}& x_{66} \\

\end{array}
\right]. \] 

From Corollary \ref{corsum}, we get that set of 3-minors that generate $J$ is
a Gr\"obner basis of $J$ with respect to any diagonal term order. Actually, this result was already known for ladder determinantal ideals (see \cite[Corollary 3.4]{Na}). Nonetheless, Corollary \ref{corsum} can be applied to more general sums of ideals. Consider for instance the ideal
$$J=I_2(X_{[12]})+I_2(X^{[12]})+I_2(X_{[56]})+I_2(X^{[56]})$$
that is the ideal generated by the 2-minors inside the below coloured region of $X$

\[ X= \left[ \: 
\begin{array}{>{\columncolor{blue!5}}c>{\columncolor{blue!5}}ccc>{\columncolor{blue!5}}c>{\columncolor{blue!5}}c}
\rowcolor{blue!5}
x_{11} & x_{12}& x_{13}& x_{14} & x_{15}& x_{16} \\
 \rowcolor{blue!5}
x_{21} & x_{22} & x_{23} & x_{24} & x_{25}& x_{26} \\
x_{31} & x_{32} & x_{33} & x_{34} & x_{35}& x_{36} \\
x_{41} & x_{42} & x_{43} & x_{44} & x_{45}& x_{46} \\
\rowcolor{blue!5}  
x_{51} & x_{52} & x_{53} & x_{54} & x_{55}& x_{56} \\
 \rowcolor{blue!5}
x_{61} & x_{62} & x_{63} & x_{64}& x_{65}& x_{66} \\
   
\end{array}
\right]. \] 

In this case, $J$ is not a ladder determinantal ideal but we can use Corollary \ref{corsum} to prove that the 2-minors that generate $J$ form a Gr\"obner basis for the ideal $J$ with respect to any diagonal term order. In fact, $J$ is a sum of ideals of the form $I_t(X_{[a,b]})$ or $I_t(X^{[c,d]})$ which are Knutson ideals from Theorem \ref{propgen2}. Then a Gr\"obner basis for $J$ is given by the union of their Gr\"obner bases.\par
Furthermore, we can also consider sums of ideals of minors of different sizes, such as
$$J=I_2(X_{[2,4]})+I_3(X^{[2,5]}).$$
In this case, $J$ is generated by the 2-minors of the blue rectangular submatrix and the 3-minors of the red rectangular submatrix illustrated below
\[ X= \left[ \: 
\begin{array}{c|ccc|cc}
\arrayrulecolor{blue}\cline{2-4}
x_{11} & x_{12}& x_{13}& x_{14} & x_{15}& x_{16} \\
 \arrayrulecolor{red}\cline{1-6}
\multicolumn{1}{|c !{\color{blue}\vrule}}{x_{21}} & x_{22} & x_{23} & x_{24} & x_{25}& \multicolumn{1}{c|}{x_{26}} \\
\multicolumn{1}{|c!{\color{blue}\vrule}}{x_{31}} & x_{32} & x_{33} & x_{34} & x_{35}& \multicolumn{1}{c|}{x_{36}} \\
\multicolumn{1}{|c!{\color{blue}\vrule}}{x_{41}} & x_{42} & x_{43} & x_{44} & x_{45}& \multicolumn{1}{c|}{x_{46}} \\
\arrayrulecolor{red}\cline{1-6}
x_{51} & x_{52} & x_{53} & x_{54} & x_{55}& x_{56} \\
x_{61} & x_{62} & x_{63} & x_{64} & x_{65}& x_{66} \\
\arrayrulecolor{blue}\cline{2-4}
   
\end{array}
\right]. \]
Again from Corollary \ref{corsum}, being $I_2(X_{[2,4]})$ and $I_3(X^{[2,5]})$ Knutson ideals, the union of their Gr\"obner bases is a Gr\"obner basis for $J$. So, a Gr\"obner basis of $J$ is given by the 2-minors of $X_{[2,4]}$ and the 3-minors of $X^{[2,5]}$.

\end{es}

 Unlike in the case of Hankel matrices, a characterization of all the ideals belonging to the family $C_f$ has not been found yet. A first step towards this result would be to understand the primary decompostion of certain sums belonging to the family. Some known results (see \cite{HS}, \cite{MR}) suggest us what these primary decompositions might be  and computer experiments seem to confirm this guess. Finding this characterization could lead to interesting properties on the Gr\"obner bases of determinantal-like ideals and it would also answer to the question asked by F.Mohammadi and J. Rhau in \cite{MR}.\\
 
\textbf{Acknowledgments}. The author is grateful to her advisor Matteo Varbaro for his valuable comments and precious advice.


\section{Knutson ideals and determinantal ideals of generic matrices}

Let $m,n$ be two positive integers with $m<n$, we will denote by $X_{mn}$ the generic matrix of size $m \times n$ with entries $x_{ij}$, that is

\[ X_{mn}=
\begin{bmatrix}
    x_{11}       & x_{12} & x_{13} & \dots & x_{1n} \\
    x_{21}       & x_{22} & x_{23} & \dots & x_{2n} \\
    x_{31}       & x_{32} & x_{33} & \dots & x_{3n} \\
   \vdots & \vdots & \vdots & \ddots & \vdots \\
    x_{m1}       & x_{m2} & x_{m3} & \dots & x_{mn}
\end{bmatrix}.
\]

Moreover, for any $1 \leq i <j \leq n$ and $1 \leq k<l \leq m$, we denote by $X_{[i,j]}^{[k,l]}$ the submatrix of $X_{mn}$ with column indices $i, i+1,\ldots, j$ and row indices $k,k+1,\ldots, l$. In the case $[k,l]=[1,m]$, we omit the superscript and we simply write $X_{[i,j]}$.\\

Given a generic matrix $X_{mn}$ and an integer $t \leq \min(m,n)$, we will denote by $I_t(X)$ the determinantal ideal in $S= \mathbb{K}[X]=\mathbb{K}[x_{i,j}\mid 1 \leq i \leq m, 1 \leq j \leq n]$ generated by all the $t$-minors of $X$.\\

We are going to prove that determinantal ideals of a generic matrix are Knutson ideals for a suitable choice of the polynomial $f$. 

%
%

\begin{thm} \label{propgen2}Let $X=X_{mn}$ be the generic matrix of size $m \times n$ with entries $x_{ij}$ and $m<n$. Consider the polynomial 
$$f= \prod_{k=0}^{m-2} \left(\det X_{[1,k+1]}^{[m-k,m]} \cdot \det X_{[n-k,n]}^{[1,k+1]} \right)\cdot \prod_{k=1}^{n-m+1} \left(\det X_{[k,m+k-1]} \right)$$  in $S=\mathbb{K}[x_{ij} \mid 1 \leq i \leq m, 1 \leq j \leq n]$. Then $I_t(X) \in \mathcal{C}_f$ for $t=1,\ldots, m$. \par
Furthermore, $I_t(X_{[a,b]})$ (respectively, $ I_t(X^{[a,b]})$) are Knutson ideals associated to $f$ for $t=1,\ldots, m$ and $0 \leq a <b \leq n$ (respectively, $0 \leq a <b \leq m$) with $b-a+1 \geq t$.

\end{thm}

A first step towards the proof of Theorem \ref{propgen2} is showing that all the ideals generated by the $t$-minors on $t$ adjacent columns are in $\mathcal{C}_f$. This fact is formally stated in the lemma below.

\begin{lem} \label{lemG1}
Let $X=X_{m \times n}$ be the generic square matrix of size $m \times n$ with entries $x_{ij}$ and let $f$ to be as in Theorem \ref{propgen2}. If we fix $t \leq m$, then:
$$I_t(X_{[i,t+i-1]}) \in \mathcal{C}_f \quad \forall i=1, \ldots, n-t+1.$$
\end{lem}

\begin{proof}
It is known (e.g. see \cite{C} and \cite{BC}) that every determinanatal ideal of a generic matrix $X=X_{m \times n}$ is prime and its height is given by the following formula:
\begin{equation}\label{ht}
\hgt (I_t (X))=(n-t+1)(m-t+1).
\end{equation} 
Therefore
$$\hgt (I_t(X_{[i,t+i-1]}))=(t+i-1-i+1-t+1)(m-t+1)=m-t+1.$$
We have three possibilities for $i$.\\

\emph{1st case}: $m-t+1 \leq i \leq n-m+1$. Then
\begin{align*}
I_t(X_{[i,t+i-1]}) \supseteq \Big( \det X_{[i,i+m-1]},  & \det X_{[i-1,i+m-2]}, \ldots,\det X_{[i-m+t,i+t-1]} \Big).
\end{align*}

\emph{2nd case}: $i \leq m-t$
\begin{align*}
I_t(X_{[i,t+i-1]}) \supseteq \Big( \det X_{[1,m]}, &  \det X_{[2,m+1]}, \ldots,\det X_{[i,i+m-1]}, \det X_{[1,t+i+1]}^{[m-t-i+2,m]}, \\
&\det X_{[1,t+i]}^{[m-t-i+1,m]} \ldots, \det X_{[1,m-1]}^{[2,m]} \Big).
\end{align*}

\emph{3nd case}: $i \geq n-m+2$
\begin{align*}
I_t(X_{[i,t+i-1]}) \supseteq \Big( \det X_{[n-m+1,n]}, &  \det X_{[n-m,n-1]}, \ldots,\det X_{[t+i-m,t+i-1]}, \det X_{[i,n]}^{[1,n-i+1]}, \\
&\det X_{[i-1,n]}^{[1,n-i+2]} \ldots, \det X_{[n-m+2,n]}^{[1,m-1]} \Big).
\end{align*}
Define $H$ to be the right hand side ideal for each of the previous cases. Note that the initial ideal of $H$ is given by some of the diagonals of the matrix $X$. Since these monomials are coprime, this ideal is a complete intersection and $$\hgt (H)=m-t+1$$ in each of the above mentioned cases. So $I_t(X_{[i,t+i-1]})$ is minimal over $H$. \par
By Definition \ref{K.I.}, $(f): J \in \mathcal{C}_f$ for every ideal $J \subseteq S$. Taking $J$ to be the principal ideal generated by the product of some of the factors of $f$, we have that all the principal ideals generated by one of the factors of $f$ are Knutson ideal associated to $f$. Being $H$ a sum of these ideals, $H \in \mathcal{C}_f$.\\ 
In conclusion, we get that $I_t(X_{[i,t+i-1]})$ is a minimal prime over an ideal of $\mathcal{C}_f$. So it is in $\mathcal{C}_f$.
\end{proof}

Using Lemma \ref{lemG1}, we can then prove Theorem \ref{propgen2}.

\begin{proof} Fix $t \in \{1,\ldots,m\}$. We want to prove that $I_t(X) \in \mathcal{C}_f$. By lemma \ref{lemG1}, we know that $I_t(X_{[1,t]}),I_t(X_{[2,t+1]}) \in \mathcal{C}_f$  and so their sum.\\
We claim that that the minimal prime decomposition of the sum is given by
$$I_t(X_{[1,t]})+I_t(X_{[2,t+1]})=I_t(X_{[1,t+1]}) \cap I_{t-1}(X_{[2,t]}).$$
To simplify the notation, we set $I_1:=I_t(X_{[1,t]}),I_2:=I_t(X_{[2,t+1]}),P_1= I_t(X_{[1,t+1]})$ and $P_2= I_{t-1} (X_{[2,t]})$.
 We want to prove that the minimal prime decomposition is given by:
 $$I_1+I_2=P_1 \cap P_2.$$
We already know that $I_1+ I_2 \subseteq P_1 \cap P_2$. Passing to the correspondent algebraic varieties, we get the reverse inclusion
 
 $$\mathcal{V} (I_1+ I_2) \supseteq \mathcal{V} (P_1 \cap P_2)$$
 
 If we prove that $\mathcal{V} (I_1+ I_2) \subseteq \mathcal{V} (P_1 \cap P_2)$, then
 
 $$\mathcal{V} (I_1+ I_2)=\mathcal{V} (P_1 \cap P_2)$$
 
 and this is equivalent to say that $ \sqrt{I_1+I_2}= \sqrt{P_1 \cap P_2}$. Since $I_1+ I_2 \in \mathcal{C}_f$, it is radical and $P_1 \cap P_2$ is radical because $P_1$ and $P_2$ are both radical ideals, then
 $$I_1+I_2=P_1 \cap P_2$$
 and we are done.\\
 For this aim, let $\mathbf{X} \in \mathcal{V}(I_1+I_2)= \mathcal{V}(I_1) \cap \mathcal{V}(I_2)$. This means that $\mathbf{X}_{[1,t]}$ and $\mathbf{X}_{[2,t+1]}$ have rank less or equal than $t-1$. Now we consider two cases:\par
 \emph{Case 1.} Suppose that $\mathbf{X}_{[2,t]}$ has rank less or equal than $t-2$. This implies that all the $(t-1) \times (t-1)$-minors corresponding to this interval vanish on $\mathbf{X}$. So $\mathbf{X} \in \mathcal{V}(P_2)$.\par
\emph{Case 2.} Suppose that $\mathbf{X}_{[2,t]}$ has full rank, namely $t-1$. Then it generates a vector space $V$ of dimension $t-1$. But by assumption, $\mathbf{X}_{[1,t]}$ and $\mathbf{X}_{[2,t+1]}$ have rank less or equal than $t-1$, so they also generate the vector space $V$. Consequently, $\mathbf{X}_{[1,t+1]}$ generates the vector space $V$ and this means that all the $t \times t$- minors of our matrix $X$ vanish on $\mathbf{X}$. Therefore we have proved that $ \mathbf{X} \in \mathcal{V}(P_1)$.\par
This proves the claim and shows that $I_t(X_{[1,t+1]}) \in \mathcal{C}_f$, being a minimal prime over a Knutson ideal.\par
In the same way, simply shifting the submatrices, we get that $I_t(X_{[k,t+k]}) \in \mathcal{C}_f$ for every $k=1, \ldots, n-t$.\par
In particular $I_t(X_{[2,t+2]}) \in \mathcal{C}_f$; therefore  the sum $I_t(X_{[1,t+1]}) +I_t(X_{[2,t+2]})$ belongs to $\mathcal{C}_f$. \par
Using a similar argument to that used to prove the claim, it can be shown that the primary decomposition of the latter sum is given by
$$I_t(X_{[1,t+1]}) +I_t(X_{[2,t+2]})=I_t(X_{[1,t+2]}) \cap I_{t-1}(X_{[2,t+1]}). $$
Therefore $I_t(X_{[1,t+2]})$ is a Knutson ideal associated to $f$.\par
Again, shifting the submatrices, the same argument shows that $I_t(X_{[k,t+k+1]}) \in \mathcal{C}_f$ for every $k=1, \ldots, n-t-1$.\par
Iterating this procedure we get that $I_t(X_{[a,b]})\in \mathcal{C}_f$ for every $1\leq a<b\leq n$ such that $b-a+1\geq t$. In particular, $I_t(X_{[1,n-1]}), I_t(X_{[2,n]}) \in \mathcal{C}_f$. Hence their sum belongs to $\mathcal{C}_f$.\par
Again, one can show that the primary decomposition of the sum is given by
$$I_t(X_{[1,n-1]}) +I_t(X_{[2,n]})=I_t(X_{[1,n]}) \cap I_{t-1}(X_{[2,n-1]}). $$
This shows that $I_t(X_{[1,n]}) \in \mathcal{C}_f$ and we are done.\par 
Notice that an identical proof shows that $I_t(X^{[a,b]}) \in \mathcal{C}_f$ for every $0 \leq a <b \leq m$ with $b-a+1 \geq t$.

\end{proof}

As an immediate consequence of the previous theorem, we get an alternative proof of $F$-purity of determinantal ideals of generic matrices.

\begin{cor} \label{corgen}
 Assume that $\mathbb{K}$ is a field of characteristic $p$ and let $X$ be a generic matrix of size $m\times n$. Then $S/I_t(X)$ is F-pure.
\end{cor}

\begin{proof}
 We may assume that $\mathbb{K}$ is a perfect field of positive characteristic. In fact, we can always reduce to this case by tensoring with the algebraic closure of $\mathbb{K}$ and the $F$-purity property descends to the non-perfect case. Using Lemma 4 in \cite{Kn}, we know that the ideal $(f)$ is compatibly split with respect to the Frobenius splitting defined by $\Tr (f^{p-1}\bullet)$ (where $f$ is taken to be as in the previous theorems). Thus all the ideals belonging to $\mathcal{C}_f$ are compatibly split with respect to the same splitting, in particular $I_t(X)$. This implies that such Frobenius splitting of $S$ provides a Frobenius splitting of $S/I_t (X)$. Being $S/I_t(X)$ $F$-split, it must be also $F$-pure.
\end{proof}

Furthermore, we obtain an interesting and useful result about Gr\"obner bases of certain sums of determinantal ideals.

\begin{cor}\label{corsum} Let $X$ be a generic matrix of size $m \times n$ and let $I$ be a sum of ideals, say $I=I_1+I_2+\ldots+I_k$, where each $I_i$ is of the form either $I_{t_i}(X_{[a_i,b_i]})$ or $I_{t_i}(X^{[a_i,b_i]})$. Then
 $$\mathcal{G}_I= \mathcal{G}_{I_1}\cup \mathcal{G}_{I_2}\cup \ldots\cup \mathcal{G}_{I_k}$$
 where $\mathcal{G}_J$ denotes a Gr\"obner basis of the ideal $J$.\par
 Furthermore, if $\mathbb{K}$ has positive characteristic, $I$ is also $F$-pure.
\end{cor}

\begin{proof}
By Theorem \ref{propgen2}, we know that $I_{t_i}(X_{[a_i,b_i]})$ and $I_{t_i}(X^{[a_i,b_i]})$ are Knutson ideals. From property $(iii)$ of Knutson ideals, we get the thesis.
\end{proof}

 \end{document}